\documentclass[12pt]{article}
\usepackage{etex}
\usepackage{graphicx}
\usepackage{subfigure}
\usepackage{algorithm,algorithmic}

\usepackage{lpic}
\usepackage{pictex, latexsym, graphicx, amsmath, amssymb, amsbsy, amsfonts, amsthm}
\usepackage[colorlinks=true,pdfnewwindow=true,urlcolor=blue,citecolor=blue,linkcolor=black]{hyperref}
\usepackage{color}

\usepackage{fullpage}

\newcommand{\beq}{\begin{equation}}
\newcommand{\eeq}{\end{equation}}
\newcommand{\bea}{\begin{eqnarray*}}
\newcommand{\eea}{\end{eqnarray*}}

\newcounter{casenum}	
\newcounter{subcasenum}	
\numberwithin{subcasenum}{casenum}
\newcounter{subsubcasenum}	
\numberwithin{subsubcasenum}{subcasenum}

\newcounter{stagenum}

\theoremstyle{plain}
\newtheorem{lemma}{Lemma} 
\newtheorem{theorem}[lemma]{Theorem}
\newtheorem{proposition}[lemma]{Proposition}

\newtheorem{observation}[lemma]{Observation}

\newtheorem{conjecture}[lemma]{Conjecture}

\theoremstyle{definition}
\newtheorem{definition}[lemma]{Definition}

\input colordvi

\newcommand{\be}{\begin{equation}}
\newcommand{\ee}{\end{equation}}
\newcommand{\bx}{\begin{bmatrix}}
\newcommand{\ex}{\end{bmatrix}}

\def\Null{{\mathsf{Null}}}
\def\True{{\mathsf{True}}}
\def\False{{\mathsf{False}}}

\def\st{\colon\,}

\def\cF{\mathcal{F}}
\def\cA{{\mathcal A}}
\def\cB{{\mathcal B}}
\def\cP{{\mathcal P}}
\def\cC{{\mathcal C}}
\def\cL{{\mathcal L}}
\def\cR{{\mathcal R}}

\def\R{\mathbb R}
\def\nCk{\binom{[n]}{k}}
\def\cutpos{\cC^+}
\def\cutneg{\cC^-}
\def\cutstar{\cC^*}
\def\genpos{\cA^+}
\def\genneg{\cA^-}
\def\branchpos{\cB^+}
\def\branchneg{\cB^-}
\def\vx{{\mathbf x}}

\def\colex{\operatorname{colex}}
\def\lex{\operatorname{lex}}
\def\meet{\wedge}
\def\join{\vee}

\def\StochasticPropagation{{\textsc{StochasticPropagation}}}
\def\BranchAndCut{{\textsc{BranchAndCut}}}

\def\PropagatePositive{{\textsc{PropagatePositive}}}
\def\PropagateNegative{{\textsc{PropagateNegative}}}

\def\SelectBranchSet{{\textsc{SelectBranchSet}}}

\def\maxkvalue{7}

\begin{document}

\title{A Branch-and-Cut Strategy for the Manickam-Mikl\'os-Singhi Conjecture}
\author{Stephen G. Hartke\thanks{Department of Mathematics, University of Nebraska--Lincoln, \texttt{hartke@math.unl.edu}. Supported by National Science Foundation Grant DMS-0914815.} \and Derrick Stolee\thanks{Department of Mathematics, University of Illinois, \texttt{stolee@illinois.edu}.}}
\date{\today}
\maketitle

\vspace{-2pc}

\begin{abstract}
	The Manickam-Mikl\'os-Singhi Conjecture states that when $n \geq 4k$, every multiset of $n$ real numbers with nonnegative total sum has at least $\binom{n-1}{k-1}$ $k$-subsets with nonnegative sum.
	We develop a branch-and-cut strategy using a linear programming formulation to show that verifying the conjecture for fixed values of $k$ is a finite problem.
	To improve our search, we develop a zero-error randomized propagation algorithm.
	Using implementations of these algorithms, we verify a stronger form of the conjecture for all $k \leq \maxkvalue$.
\end{abstract}

\section{Introduction}

Given a sequence $x_1,\dots,x_n$ of real numbers with nonnegative sum $\sum_{i=1}^n x_i$, it is natural to ask how many partial sums must be nonnegative.
Bier and Manickam~\cite{BM} showed that every nonnegative sum $\sum_{i=1}^n x_i$ has at least $2^{n-1}$ nonnegative partial sums.
Manickam, Mikl\'os, and Singhi~\cite{MM,MS} considered the situation when each partial sum has exactly $k$ terms (we call such partial sums \emph{$k$-sums}), and they conjectured there are many nonnegative $k$-sums when $n$ is sufficiently large.

\begin{conjecture}[Manickam-Mikl\'os-Singhi~\cite{MM,MS}]
\label{conj:mms}
Let $k \geq 2$ and $n \geq 4k$.
For all $x_1,\dots,x_n \in \R$ such that $\sum_{i=1}^n x_i \geq 0$, there are at least $\binom{n-1}{k-1}$ subsets $S \subset [n]$ with $|S| = k$ such that  $\sum_{i\in S} x_i \geq 0$.
\end{conjecture}
The bound $\binom{n-1}{k-1}$ is sharp since the example $x_1 = n-1$ and $x_i = -1$ for $i \neq 1$ has sum zero and a $k$-sum is nonnegative exactly when it contains $x_1$.
The bound $n \geq 4k$ is not necessarily sharp.

In this paper, we develop a computational method to solve fixed cases of the conjecture based on a poset of $k$-sums and a linear programming instance.
In particular, we prove a stronger statement than the conjecture for all $k \leq \maxkvalue$.
This leads us to formulate a stronger form of the conjecture (see Conjecture~\ref{conj:ours}).

Conjecture~\ref{conj:mms} is trivial for $k = 1$ and is a simple exercise for $k = 2$.
Marino and Chiaselotti~\cite{MC} proved the conjecture for $k = 3$.
Bier and Manickam~\cite{BM} showed there exists a minimum integer $f(k)$ such that for all $n \geq f(k)$, there are at least $\binom{n-1}{k-1}$ nonnegative $k$-sums in any nonnegative sum of $n$ terms.
Subsequent results~\cite{B03, BM, C02, CIM, T12} give several exponential upper bounds on $f(k)$.
Alon, Huang, and Sudakov \cite{AHS} showed the polynomial bound $f(k) \leq \min\{ 33k^2, 2k^3\}$.
Chowdhury~\cite{C} proved $f(3) = 11$ and $f(4) \leq 24$.

We prove $f(4) = 14$, $f(5) = 17$, $f(6) = 20$, and $f(7) = 23$.
For $k \leq \maxkvalue$ and $n < f(k)$, we find that the nonnegative sums with the fewest nonnegative $k$-sums have the following structure: for positive integers $a$ and $b$ summing to $n$, let $x_1 = \cdots = x_b = a$ and $x_{b+1} = \cdots = x_n = -b$.
When $3k < n < f(k)$, the extremal examples have $a = 3$, $b = n - 3$, and $\binom{n-3}{k}$ nonnegative $k$-sums.
This leads us to make the following conjecture.

\begin{conjecture}\label{conj:ours}
	Let $N_k$ be the smallest integer such that $\binom{N_k-3}{k} \geq \binom{N_k-1}{k-1}$.
	Then $f(k) = N_k$, 
	and hence\footnote{This limit was computed by first simplifying $\binom{x-3}{k}-\binom{x-1}{k-1}=0$ to $(x-k)(x-k-1)(x-k-2) - k(x-1)(x-2)=0$ and taking the real root for $x$ of the resulting cubic.}
	\[
		\lim_{k\to\infty} \frac{f(k)}{k}  	
			= \lim_{k\to\infty} \frac{N_k}{k} 
			\approx 3.147899.
		\]
\end{conjecture}

In the next section we develop some necessary notation and preliminary results.
In Section~\ref{sec:poset}, we place a symmetry-breaking constraint on the vectors and investigate a natural poset among the $k$-sums.
In Section~\ref{sec:algorithm}, we develop our branch-and-cut strategy for verifying the conjecture.
We then modify this algorithm by adding a randomized propagation algorithm in Section~\ref{sec:branchless}.
In Section~\ref{sec:strong} we discuss stronger forms of the conjecture.

\subsection{Preliminaries}

Let $a_1,\dots, a_\ell$ be real numbers and $e_1,\dots,e_\ell$ be positive integers.
We denote by $a_1^{e_1}\ a_2^{e_2}\ \dots a_\ell^{e_\ell}$ the vector in $\R^n$ with the first $e_1$ coordinates equal to $a_1$, the next $e_2$ coordinates equal to $a_2$, and so on until the last $e_\ell$ coordinates are equal to $a_\ell$.
Using this notation, the vector $(n-1)^1\ (-1)^{n-1}$ is conjectured to achieve the minimum number of nonnegative $k$-sets when $n \geq 4k$.

For a vector $\vx = (x_1,\dots,x_n) \in \R^n$ and a $k$-set $S$, define $\sigma_S(\vx) = \sum_{i\in S} x_i$.
The number of nonnegative $k$-sums of $\vx$ is $s_k(\vx) = \left|\left\{ S \in \nCk : \sigma_S(\vx) \geq 0\right\}\right|$.
For small values of $k$ and $n$, we will compute the minimum integer $g(n,k)$ such that all vectors $\vx \in \R^n$ with nonnegative sum have $s_k(\vx) \geq g(n,k)$.
By the sharpness example, $g(n,k) \leq \binom{n-1}{k-1}$ always.
Define the \emph{deficiency function} as $\hat g(n,k) = \binom{n-1}{k-1} - g(n,k)$.

Since removing the least coordinate from a vector  $\vx \in \R^n$ results in an $(n-1)$-coordinate vector with nonnegative sum, the function $g(n,k)$ is nondecreasing in $n$.
Bier and Manickam~\cite{BM} proved that $\hat g(n,k) = 0$ when $k$ divides $n$ by using a theorem of Baranyai~\cite{Baranyai}.
The following lemma of Chowdhury~\cite{C} reduces the problem of computing $f(k)$ to verifying $\hat g(n,k) = 0$ for a finite number of values $n$.

\begin{lemma}[Chowdhury~\cite{C}]\label{lma:remainders}
If $\hat g(n,k) = 0$, then $\hat g(n+k,k) = 0$.
\end{lemma}

This lemma allows the computation of $f(k)$ by first computing a finite number of values of $g(n,k)$.
We develop a finite process to verify $g(n,k) \geq t$ for fixed values of $k$ and $n$, and hence by Lemma~\ref{lma:remainders} we determine $f(k)$ when we find that $g(n,k) \geq \binom{n-1}{k-1}$ for $k$ consecutive values of $n$.
Our computation to verify $g(n,k) \geq t$ for fixed integers $n$, $k$, and $t$ is designed as a search for a vector $\vx \in \R^n$ with $s_k(\vx) < t$.
By constraining the order of the coordinates in a vector $\vx \in \R^n$, we are able to make significant deductions about the sign of the $k$-sums of $\vx$, which makes the computation very efficient.
We first discuss this constraint on the coordinates of $\vx$ in Section~\ref{sec:poset}, and then discuss the algorithm in Section~\ref{sec:algorithm}.

\section{The Shift Poset}\label{sec:poset}

To better control the vectors $\vx \in \R^n$ with nonnegative sum, we make a simple ordering constraint that adds significant structure.
Let $F_n$ be the set of vectors $\vx \in \R^n$ with nonnegative sum and with the sequence $x_i$ nondecreasing. 
That is,
\[
	F_n = \left\{ \vx  \in \R^n : \sum_{i=1}^n x_i \geq 0\text{ and } x_1 \geq x_2 \geq \cdots \geq x_n \right\}.
\]	

Denote by $[n]$ the set $\{1,\dots,n\}$ and by $\nCk$ the collection of $k$-subsets of $[n]$.
For two $k$-sets $S = \{ i_1 < i_2 < \cdots < i_k\}$ and $T = \{ j_1 < j_2 < \cdots < j_k \}$, let $S \succeq T$ if and only if $i_\ell \leq j_\ell$ for all $\ell \in \{1,\dots,k\}$.
We call this poset over $\nCk$ the \emph{shift poset}.
Since we write $x_1,\dots,x_n$ from left to right in nondecreasing order and associate a $k$-set $S$ as with the values $x_i$ for $i \in S$, then when $S \succeq T$ we say $S$ is \emph{to the left} of $T$ and $T$ is \emph{to the right} of $S$.
Observe that a relation $S \succeq T$ is a cover relation if and only if $S \setminus T = \{ i\}$, $T \setminus S = \{ j\}$, and $i = j-1$.

Since we restrict vectors in $F_n$ to have nondecreasing values as the index increases, $S \succeq T$ implies $\sigma_S(\vx) \geq \sigma_T(\vx)$ for all $\vx \in F_n$.
Therefore, if $\sigma_T(\vx)$ is nonnegative, then so is $\sigma_S(\vx)$.
Similarly, if $\sigma_S(\vx)$ is strictly negative, then so is $\sigma_T(\vx)$.

Given $S \in \nCk$, let the \emph{left shift} of $S$ be $\cL_k(S) = \{ T \in \nCk \st T \succeq S\}$ and the \emph{right shift} of $S$ be $\cR_k(S) = \{ T \in \nCk \st S \succeq T\}$.
Given a set $\cA \subseteq \nCk$, the \emph{left closure} of $\cA$, denoted $\cL_k(\cA)$, is the union of all families $\cL_k(S)$ over all sets $S \in \cA$.
Similarly, the \emph{right closure} of $\cA$, denoted $\cR_k(\cA)$,  is the union of all families $\cR_k(S)$ over all sets $S \in \cA$.

The shift poset has many interesting properties, including the fact that it is a lattice, which will be critical in later calculations.
Fix any list $\cF = \{ S_1, \dots, S_\ell\}$ of $k$-sets where $S_j = \{ i_{j,1} < \cdots < i_{j,k}\}$.
The \emph{meet} of $\cF$, denoted $\meet \cF$ or $\meet_{j=1}^\ell S_j$, is the $k$-set $T$ such that for all $T'$ where $S_j \succeq T'$ for all $j \in [\ell]$, then $T \succeq T'$.
The \emph{join} of $\cF$, denoted $\join \cF$ or $\meet_{j=1}^\ell S_j$, is the $k$-set $T$ such that for all $T'$ where $T' \succeq S_j$ for all $j \in [\ell]$, then $T' \succeq T$.
The meet and join can be computed as
\begin{align*}
	\meet \cF = \meet_{j=1}^\ell S_j &= \big\{ \max \{ i_{j,t} : j \in [\ell]\} : t \in [k] \big\},\\
	\join \cF = \join_{j=1}^\ell S_j  &= \big\{ \min \{ i_{j,t} : j \in [\ell]\} : t \in [k] \big\}.
\end{align*}

Observe that the $k$-sets to the left of all $S_1,\dots,S_\ell$ are exactly the $k$-sets to the left of the join $\join_{j=1}^\ell S_j$, and the $k$-sets to the right of all $S_1,\dots,S_\ell$ are exactly the $k$-sets to the right of the meet $\meet_{j=1}^\ell S_j$.
That is,
\[
	\bigcap_{j=1}^\ell \cL_k(S_j) = \cL\left(\join_{j=1}^\ell S_j\right),\qquad 
	\bigcap_{j=1}^\ell \cR_k(S_j) = \cR\left(\meet_{j=1}^\ell S_j\right).
\]

The shift poset preserves order with two standard linear orders of $k$-sets: the \emph{lexicographic} and \emph{colexicographic} orders.
Let $\lex(S)$ and $\colex(S)$ be the \emph{lexicographic} and \emph{colexicographic} rank, respectively, of a $k$-set $S = \{i_1 < \dots < i_k\}\subseteq \{1,\dots,n\}$, defined as 
\[
\lex(S) = \sum_{\ell=1}^k \sum_{j=i_{\ell-1}+1}^{i_\ell - 1} \binom{n - j}{k-\ell}, 
	\qquad	
\colex(S) = \sum_{\ell =1}^k \binom{i_\ell - 1}{\ell},
\]
where $i_0 = 0$ by convention.
These functions are bijections from $\nCk$ to the set $\{0, \dots, \binom{n}{k}-1\}$ with the property that when $S \succeq T$ we also have $\lex(S) \leq \lex(T)$ and $\colex(S) \leq \colex(T)$.
In addition to these ranking functions, the lexicographic and colexicographic orders also have efficient \emph{successor} and \emph{predecessor} calculations.
We use these methods to iterate over all $k$-sets to find $\succeq$-maximal or $\succeq$-minimal elements.

Finally, we note that the shift poset is isomorphic to the \emph{dominance poset} over compositions of $n+1$ into $k+1$ positive parts.
Given two compositions of $n+1$ into $k+1$ positive parts, where $A = \{a_1,\dots, a_{k+1}\}$ and $B = \{ b_1,\dots, b_{k+1}\}$,  $A \trianglelefteq B$ in the dominance order when $\sum_{\ell=1}^t a_\ell \leq \sum_{\ell=1}^t b_\ell$ for all $t \in [k]$.
Note that the dominance order, also called the \emph{majorization order}, is usually defined over partitions, but the defining inequalities are the same for compositions (see~\cite{Stanley}).

For $S = \{i_1 < \cdots < i_k\}$, let $A_S = \{ a_1, \dots, a_{k+1}\}$ where $a_1 = i_1$, $a_\ell = i_\ell -i_{\ell-1}$ for $2 \leq \ell \leq k$, and $a_{k+1} = n+1-i_k$.
The image $A_S$ has the property that $\sum_{\ell=1}^t a_\ell = i_t$ for $t \leq k$ and $\sum_{\ell=1}^{k+1} a_\ell = n + 1$.
Therefore, if $S \succeq T$ in the shift poset, then $A_S \trianglelefteq A_T$ in the dominance order.

\subsection{Counting Shifts}

The functions $L_k(S) = |\cL_k(S)|$ and $R_k(S) = |\cR_k(S)|$ count the size of the left and right shifts, respectively.
These shift functions are quickly computable using a recursive formula due to Chowdhury~\cite{C} which we prove for completeness.
If $S = \{i_1,\dots, i_k\}$ and $j \leq k$, let $S_j = \{i_1,\dots, i_j\}$.
We define $L_j(S) = |\cL_j(S_j)|$.

\begin{proposition}[Chowdhury~\cite{C}]
	For a $k$-set $S = \{ i_1 < i_2 < \cdots < i_k \} \in \nCk$,
	\[
	L_k(S) = \binom{i_k}{k} - \binom{i_k-i_1}{k} - \sum_{\ell=1}^{k-2} L_\ell(S)\binom{i_k-i_{\ell+1}}{k-\ell}.
	\]
\end{proposition}

\begin{proof}
Observe that $\cL_k(S) \subseteq \binom{[i_k]}{k}$.
We will count the sets $T = \{ j_1 < \cdots < j_k\}$ in $\binom{[i_k]}{k}\setminus \cL_k(S)$ as $\binom{i_k-i_1}{k} + \sum_{\ell=1}^{k-2} L_\ell(S_\ell)\binom{i_k-i_{\ell+1}}{k-\ell}$.
Every such set $T$ has some minimum parameter $\ell \geq 0$ such that $j_{\ell+1} > i_{\ell+1}$.
If $\ell = 0$, then the minimum of $T$ is strictly larger than $i_1$ and there are exactly $\binom{i_k-i_1}{k}$ such sets. 
When $\ell \geq 1$, we have $j_t \leq i_t$ for all $t \leq \ell$.
These sets can be constructed from the $L_\ell(S_\ell)$ $\ell$-subsets of $[n]$ that are to the left of $S_\ell$ and extended by exactly $\binom{i_k-i_{\ell+1}}{k-\ell}$ $(k-\ell$)-subsets of $\{i_{\ell+1}+1, \dots, i_k\}$.
\end{proof}

Observe that if $S = \{ i_1 < \cdots < i_k\}$, then by symmetry for $T = \{ n + 1 - i_\ell : \ell \in [k] \}$ we have $R_k(S) = L_k(T)$.
While computing $L_k(S)$ for all $k$-sets, we first compute and store the values of $L_j(S_j)$ for $1\leq j < k$.

\subsection{Required Negative Sets}

We use the shift function $L_k(S)$ to determine that certain $k$-sums must be negative in order to avoid $t$ nonnegative $k$-sums.

\begin{observation}\label{obs:forcenegative}
	Let $n,k, t$ be integers.
	If $S \in \nCk$ has $L_k(S) \geq t$, then every vector $\vx \in F_n$ with $s_k(\vx) < t$ has $\sigma_S(\vx) < 0$.
\end{observation}

\begin{lemma}\label{lma:baranyaitrick}
	Fix integers $n, k, t$ with $t \leq \binom{n-1}{k-1}$ and a $k$-set $S \in \nCk$ with $1 \in S$.
	If $L_k(S) + g(n-k,k)  \geq t$, then every vector $\vx \in F_n$ with $s_k(\vx) < t$ has $\sigma_S(\vx) < 0$. 
\end{lemma}

\begin{proof}
	Fix a vector $\vx \in F_n$ with $s_k(\vx) < t$.
	The $k$-set $T = \{1, n-k+2, \dots, n\}$ has $L_k(T) = \binom{n-1}{k-1}$, so by Observation~\ref{obs:forcenegative},  $\sigma_T(\vx) < 0$.
	Since $\sum_{i=1}^n x_i \geq 0$, we have $\sum_{i=2}^{n-k+1} x_i = \sum_{i \notin T} x_i \geq 0$.
	Let $\cA$ be the family of nonnegative $k$-sums over $x_2,\dots,x_{n-k}$.
	By the definition of $g$, $|\cA| \geq g(n-k,k)$.
	For all $Q \in \cA$, the minimum element of $Q$ is at least 2 and hence $Q$ is not to the left of $S$, since the minimum element of $S$ is $1$.
	Therefore, if $\sigma_S(\vx) \geq 0$, then $s_k(\vx) \geq |\cL_k(S)| + |\cA| \geq L_k(S) + g(n-k,k)\geq t$, a contradiction.
\end{proof}	


Similar inferences, discussed in Section~\ref{sec:propagation}, will be important to the performance of our algorithms.

\section{The Branch-and-Cut Method}\label{sec:algorithm}

Our method to verify the conjecture centers on searching for vectors $\vx \in F_n$ which are counterexamples (i.e. $s_k(\vx) < \binom{n-1}{k-1}$).
In addition, for values of $n$ less than $f(k)$, we will search for vectors with the fewest number of nonnegative $k$-sets.
For integers $n, k,$ and $t$, we say a vector $\vx \in F_n$ is \emph{$(n,k,t)$-bad} if $s_k(\vx) < t$.
Using the shift poset, we will determine properties of such an $(n,k,t)$-bad vector and use that to guide our search.

For a vector $\vx \in F_n$, the $k$-subsets of $[n]$ are partitioned into two parts by whether the associated $k$-sums are nonnegative or strictly negative.
Let $\cutpos_{\vx}$ contain the $k$-sets $S$ with $\sigma_S(\vx) \geq 0$ and $\cutneg_{\vx}$ contain the the $k$-sets $T$ with $\sigma_T(\vx) < 0$.
Observe $\cutpos_{\vx} = \cL_k(\cutpos_{\vx})$  and $\cutneg_{\vx} = \cR_k(\cutneg_{\vx})$.
Our computational method focuses on testing whether vectors $\vx \in F_n$ exist, given certain contraints on $\cutpos_{\vx}$ and $\cutneg_{\vx}$.

\begin{definition}\label{def:lp}
Given families $\genpos$ and $\genneg$ of $k$-sets, let the linear program $P(k,n,\genpos, \genneg)$ be defined as
\begin{align*}
	P(k, n, \genpos, \genneg): \quad\qquad\qquad
	\text{minimize}\qquad x_1&\\
	\text{subject to } \sum_{i=1}^n x_i &\geq 0 \\
	x_i - x_{i+1} &\geq 0 \qquad \forall i\in \{1,\dots,n-1\}\\
	\sum_{i\in S} x_i &\geq 0 \qquad \forall S \in \genpos\\
	\sum_{i\in T} x_i &\leq -1 \quad \forall T \in \genneg\\
	x_1,\dots,x_n &\in \R
\end{align*}
\end{definition}

A vector $\vx \in F_n$ has $\cutpos_{\vx} \supseteq \cL_k(\genpos)$ and $\cutneg_{\vx} \supseteq \cR_k(\genneg)$ if and only if an appropriate scalar multiple of $\vx$ is a feasible solution to $P(n,k,\genpos,\genneg)$.

Due to the following observations, we can verify that $s_k(\vx) \geq t$ for the infinite set of vectors $\vx$ in $F_n$ by searching for partitions of $\nCk$ that correspond to the nonnegative and negative $k$-sums of $\vx$.

\begin{observation}\label{obs:completeness}
	If $\cL_k(\genpos) \cup \cR_k(\genneg) = \nCk$, then every feasible solution $\vx$ to $P(n,k,\genpos,\genneg)$ has $s_k(\vx) = |\cL_k(\genpos)|$.
	In particular, a feasible solution to $P(n,k,\genpos,\genneg)$ is $(n,k,t)$-bad if and only if $|\cL_k(\genpos)| < t$.
\end{observation}

\begin{observation}\label{obs:pruning}
	If $P(n,k,\genpos,\genneg)$ is infeasible, then no vector $\vx \in F_n$ has $\cutpos_{\vx} \supseteq \cL_k(\genpos)$ and $\cutneg_{\vx}\supseteq \cR_k(\genneg)$.
\end{observation}

\begin{observation}\label{obs:propagate}
	If $|\cL_k(\genpos)| \geq t$, then $s_k(\vx) \geq t$ for all feasible solutions to $P(n,k,\genpos,\genneg)$.
\end{observation}

By searching for sets $\genpos, \genneg$ such that a solution to $P(n,k,\genpos,\genneg)$ has $s_k(\vx) < t$, we either find these vectors or determine that none exist.
Thus, determining $g(n,k)$ is a finite problem, and by Lemma~\ref{lma:remainders} determining $f(k)$ is a finite problem.

\subsection{The Search Strategy}

Fix $n$, $k$, and $t \leq \binom{n-1}{k-1}$.
We will search for sets $\genpos$ and $\genneg$ such that the solution $\vx$ to $P(n,k,\genpos,\genneg)$ has $s_k(\vx) < t$.
If $\cL_k(\genpos) \cup \cR_k(\genneg)$ is a partition of $\nCk$ with $|\cL_k(\genpos)| < t$, then every feasible solution to $P(n,k,\genpos,\genneg)$ is $(n,k,t)$-bad.
The reason we use the objective of minimizing $x_1$ in the linear program is to try and distance the optimal solution from the conjectured sharp example of $x_1$ being a large positive value and $x_2,\dots,x_n$ being small negative values.
In practice, when an $(n,k,t)$-bad vector exists we discover it before $\cL_k(\genpos)$ and $\cR_k(\genneg)$ partition $\nCk$.

During the algorithm, we will store two collections of $k$-sets, $\branchpos$ and $\branchneg$,  called \emph{branch sets}.
We initialize $\branchpos = \branchneg = \varnothing$ and always $\branchpos \subseteq \genpos$ and $\branchneg \subseteq \genneg$.
The difference between $\branchneg$ and $\genneg$, for instance, is that the sets in $\branchneg$ are \emph{chosen} to have negative sum, but the sets in $\genneg \setminus \branchneg$ are sets that \emph{must} have negative sum for all $(n,k,t)$-bad vectors $\vx$ with $\cutpos_\vx \supseteq \branchpos$ and $\cutneg_\vx\supseteq \branchneg$.

The search procedure \BranchAndCut\ (Algorithm~\ref{alg:formulationiirecurse}) is recursive, taking parameters $n, k, t, \branchpos, \branchneg, \cutstar$ under the requirement that $\cL_k(\branchpos) \cup \cR_k(\branchneg) \cup \cutstar$ is a partition of $\nCk$.
The collection $\cutstar = \nCk \setminus (\cL_k(\branchpos) \cup \cR_k(\branchneg))$ contains the $k$-sets $S$ where $\sigma_S(\vx)$ is not immediately decided by the constraints on $\branchpos$ and $\branchneg$.

We will refer to each recursive call to \BranchAndCut\ as a \emph{search node}, and at each search node we perform the following three actions:

\begin{enumerate}
	\item Determine the $\succeq$-maximal sets $S \in \cutstar$ such that all vectors $\vx$ with $\cutpos_\vx \supseteq \branchpos \cup \{S\}$ have $s_k(\vx) \geq t$ and add these sets to $\genneg$.
	\item Test that the linear program $P(n,k,\branchpos,\genneg)$ is feasible (prune if infeasible).
	\item Select a set $S \in \cutstar$ and branch on the choice of placing $S$ in $\branchpos$ or $\branchneg$, creating two new search nodes.
\end{enumerate}

\begin{algorithm}[tp]
\begin{algorithmic}
	\REQUIRE $\cL_k(\branchpos)$, $\cR_k(\branchneg)$, and $\cutstar$ partition $\nCk$.
	\ENSURE Returns an $(n,k,t)$-bad vector $\vx$ with $\cutpos_\vx \supseteq \branchpos$ and $\cutneg_{\vx}\supseteq \branchneg$ if it exists.
	\IF{$|\cL(\branchpos)| \geq t$}
		\RETURN $\Null$
	\ENDIF
	\STATE $\genpos, \genneg, \cutstar \leftarrow$ \PropagateNegative$(n,k,t,\branchpos,\branchneg,\cutstar)$\qquad(\emph{See Algorithm \ref{alg:propagation}})
	\STATE Find optimal solution $\vx$ to $P(n,k,\genpos,\genneg)$
	\IF{$\vx \equiv \Null$ \textbf{or} $s_k(\vx) < t$}
		\STATE \emph{The linear program is infeasible or found a counterexample.}
		\RETURN $\vx$
	\ENDIF
	\STATE $S \leftarrow$ \SelectBranchSet$(\genpos, \genneg, \cutstar)$
	\STATE $\branchpos_1  \leftarrow \branchpos$,\quad $\branchneg_1  \leftarrow \branchneg \cup \{ S \}$,\quad $\cutstar_1  \leftarrow \cutstar_0 \setminus \cR_k(S)$
	\STATE $\vx \leftarrow$ \BranchAndCut$(n, k, t, \branchpos_1, \branchneg_1, \cutstar_1)$
	\IF{$\vx \equiv \Null$}
	\STATE $\branchpos_2 \leftarrow \branchpos \cup \{ S \}$,\quad $\branchneg_2 \leftarrow \branchneg$,\quad $\cutstar_2  \leftarrow \cutstar \setminus \cL_k(S)$
	\STATE $\vx \leftarrow$ \BranchAndCut$(n,k,t,\branchpos_2, \branchneg_2, \cutstar_2)$
	\ENDIF
	\RETURN $\vx$
\end{algorithmic}
\caption{\label{alg:formulationiirecurse}The recursive algorithm \BranchAndCut$(n, k, t, \branchpos, \branchneg, \cutstar)$.}
\end{algorithm}

One crucial step for the first action performed at each node is computing the number of $k$-sets in $\cutstar$ that are also to the left of a set $S$ or to the right of a set $S$.
Define the functions $L_k^*(S) = |\cL_k(S) \setminus \cL_k(\genpos)|$ and $R_k^*(S) = |\cR_k(S) \setminus \cR_k(\genneg)|$.

The algorithm \SelectBranchSet\ is a method for selecting a set $S \in \cutstar$ for the branching step is called a \emph{branching rule}.
A good branching rule can significantly reduce the number of visited search nodes.
However, it is difficult to predict which set is the best choice.
Any branching rule that selects a $k$-set $S$ from $\cutstar$ will provide a \emph{correct} algorithm,  but these choices can greatly change the size and structure of the search tree.
Based on our experiments and choices of several branching rules, we found selecting a $k$-set $S$ that maximizes $\min\{L_k^*(S), R_k^*(S)\}$ was most effective.
This rule ensured that both branches removed as many sets from $\cutstar$ as possible.

In Section~\ref{sec:propagation}, we define the algorithm \PropagateNegative\ (Algorithm~\ref{alg:propagation}).
Before that, we must discuss how to efficiently compute $L_k^*(S)$ and $R_k^*(S)$.


\subsection{Computing Intersections with $\cutstar$}\label{sec:starfunctions}

For this discussion, fix two families $\genpos$ and $\genneg$ with $\cutstar = \nCk \setminus (\cL_k(\genpos) \cup \cR_k(\genneg))$.
We will use three methods to compute $|\cL_k(S) \cap \cutstar|$ and $|\cR_k(S) \cap \cutstar|$: breadth-first-search, inclusion-exclusion, and iterative updates.
In all techniques, we will store markers for whether a set $T$ is in $\cL_k(\genpos)$, $\cR_k(\genneg)$, or $\cutstar$.

\vspace{0.5em}
\noindent{\bf Breadth-First-Search.} 
When $k$ is small, it is reasonable to store the Hasse digraph (an edge $S \leftarrow T$ exists if $S \succeq T$ is a cover relation) of the shift poset in memory and use breadth-first-search to count the number of sets to the left of $S$ which are in $\cutstar$.

\vspace{0.5em}
\noindent{\bf Inclusion-Exclusion.}
We can use the lattice structure of the shift poset to compute the functions $L_k^*(S)$ and $R_k^*(S)$.
First, we compute the values of $L_k^*(S)$ with increasing colex rank and values of $R_k^*(S)$ with decreasing colex rank.
With this order, we have access to $L_k^*(T)$ whenever $T \succeq S$ or $R_k^*(T)$ whenever $S \succeq T$.
Let $\cF_L(S)$ be the collection of sets $T$ that cover $S$, and let $\cF_R(S)$ be the collection of sets $T$ that $S$ covers.
Then, 
\begin{align*}
	L_k^*(S) &= 1 + \sum_{\emptyset \neq \cA \subseteq \cF_L(S)} (-1)^{|\cA|+1} L_k^*(\join \cA)\\
	R_k^*(S) &= 1 + \sum_{\emptyset \neq \cA \subseteq \cF_R(S)} (-1)^{|\cA|+1} R_k^*(\meet \cA)
\end{align*}

When $k \leq 4$, the breadth-first-search strategy is faster than the inclusion-exclusion strategy.
However, for $k \geq 6$, the inclusion-exclusion strategy is the only method that is tractable.
For $k = 5$, the two strategies are too similar to demonstrate a clear preference.

As $k$ grows, the number of terms of the inclusion-exclusion sum grows exponentially, leading to a large amount of computation required for every set $S \in \cutstar$.
Our next method computes the values of $L_k^*(S)$ when a small change has been made to $\branchpos$ using a smaller amount of computation for every set $S \in \cutstar$.

\vspace{0.5em}
\noindent{\bf Updates with $\branchpos$.}
Observe that sets are added to $\genpos$ only when placing a set into $\branchpos$.
Therefore, there are many fewer sets in $\genpos$ than in $\genneg$.
Further, if the most recent branch selected a $k$-sum to be negative, then the values $L_k^*(S)$ did not change for any sets still in $\cutstar$.
Therefore, we need only update the values of $L_k^*(T)$ for $T \in \cutstar$ after branching on a set $S$ and placing $S$ in $\branchpos$.
For all such $T$, observe that
\[
	L_k^*(T) \leftarrow L_k^*(T) - L_k^*(T \join S)
\]
properly updates the value of $L_k^*(T)$ when adding the set $S$ to $\branchpos$. 
Specifically, the sets that are removed from $\cL_k(T) \setminus \cL(\genpos)$ to $\cL_k(T) \setminus \cL_k(\genpos \cup \{S\})$ are exactly those to the left of both $T$ and $S$ but not to the left of any elements in $\genpos$.
Such $k$-sets are exactly those in $\cL_k(T\join S) \setminus \cL(\genpos)$, which are counted by $L_k^*(T \join S)$.

This method is not efficient for computing $R_k^*(S)$, since too many sets are being added to $\genneg$ during the propagation step (see the next section). 
Thus, we use this method to update $L_k^*(S)$ only when a single set has been added to $\branchpos$.
If we need to recompute all values of $L_k^*(S)$ after a significant change to $\branchpos$, we use the inclusion-exclusion method.

\subsection{Propagation}\label{sec:propagation}

Given a set of branch sets $\branchpos$ and $\branchneg$, we want to determine which sets $S \in \cutstar$ must have $\sigma_S(\vx) < 0$ for all $(n,k,t)$-bad vectors $\vx$ which are feasible in $P(n,k,\genpos,\genneg)$.

Recall that by Observation~\ref{obs:forcenegative} and Lemma~\ref{lma:baranyaitrick}, any set $S$ with $L_k(S) \geq t$, or $1 \in S$ and $L_k(S) + g(n-k,k) \geq t$, has $\sigma_S(\vx) < 0$ for any $(n,k,t)$-bad vector $\vx$.
This applies regardless of the previous choices of sets placed in $\branchpos$.
When $\branchpos$ is non-empty and we have access to the function $L_k^*(S)$, we may be able to find more sets where $\sigma_S(\vx) < 0$ for any $(n,k,t)$-bad vector $\vx$ with $\cutpos_\vx \supseteq \branchpos$.

\begin{lemma}\label{lma:propagatenegative}
	If $S$ is a $k$-set with $L_k^*(S) + |\cL_k(\branchpos)| \geq t$, then all feasible solutions $\vx$ to $P(n, k, \branchpos\cup\{S\}, \branchneg)$ have $s_k(\vx) \geq t$.
\end{lemma}

\begin{proof}
Observe that such vectors $\vx$ have $\cutpos_{\vx} \supseteq \cL_k(\branchpos) \cup \cL_k(S)$ and $|\cL_k(\branchpos)\cup \cL_k(S)| = |\cL_k(\branchpos)| + |\cL_k(S)\setminus \cL_k(\branchpos)| = |\cL_k(\branchpos)| + L_k^*(S) \geq t$.
\end{proof}

The \PropagateNegative\ algorithm (Algorithm~\ref{alg:propagation}) iterates on all sets $S$ in $\cutstar$ using lexicographic order and whenever a set $S$ satisfies the hypotheses of Lemma~\ref{lma:baranyaitrick} or Lemma~\ref{lma:propagatenegative} the set $S$ is added to $\genneg$.
Since all sets $T \succeq S$ have $T \leq_{\text{lex}} S$, the set $T$ was considered before $S$ and did not satisfy the condition for addition to $\genneg$.
Hence, we construct $\genneg$ as a $\succeq$-maximal set (other than possible comparisons within $\branchneg$).

\begin{algorithm}[tp]
\begin{algorithmic}
\REQUIRE $\cL_k(\branchpos),\cR_k(\branchneg)$, and $\cutstar$ partition $\nCk$ and $|\cL_k(\branchpos)| < t$.
\ENSURE Returns $\branchpos$, $\genneg_0$, and $\cutstar_0$ such that any $(n,k,t)$-bad vector $\vx$ with $\cutpos_\vx \supseteq \branchpos$ and $\cutneg_\vx \supseteq \branchneg$ also has $\cutneg_\vx \supseteq \genneg_0$, and $\cutstar_0 = \nCk \setminus (\cL_k(\branchpos)\cup \cR_k(\genneg_0))$.

\vspace{1em}
\STATE $\genneg_0 \leftarrow \branchneg$,\quad $\cutstar_0 \leftarrow \cutstar$

\FOR{{\bf all} sets $S \in \cutstar_0$ (in lex order)}
	\IF{$1 \in S$ \textbf{and} $L_k(S) + g(n-k,k) \geq t$}
		\STATE $\genneg_0 \leftarrow \genneg_0 \cup \{S\}$
		\STATE $\cutstar_0 \leftarrow \cutstar_0 \setminus \cR_k(S)$
	\ENDIF
	\IF{$L_k^*(S) + |\cL_k(\branchpos)| \geq t$}
		\STATE $\genneg_0 \leftarrow \genneg_0 \cup \{S\}$
		\STATE $\cutstar_0 \leftarrow \cutstar_0 \setminus \cR_k(S)$
	\ENDIF
\ENDFOR

\RETURN $\branchpos, \genneg_0, \cutstar_0$
\end{algorithmic}
\caption{\label{alg:propagation}\PropagateNegative$(n, k, t, \branchpos, \branchneg,\cutstar)$}
\end{algorithm}

\subsection{Results Using \BranchAndCut}

Using \BranchAndCut\ (Algorithm~\ref{alg:formulationiirecurse}), we computed $g(n,k)$ for $k \in \{4,5,6\}$ and $k < n \leq 4k+1$, and we verified that $f(k) = 3k+2$ for $k \in \{4,5,6\}$.

\begin{theorem}\label{thm:branching}
	$f(4) = 14$, $f(5) = 17$, and $f(6) = 20$.
\end{theorem}

\begin{proof}[Proof outline.]
	By Lemma~\ref{lma:remainders}, we must only verify the values $n$ where $f(k) \leq n \leq f(k) + k - 1$.
	For all $n \leq 3k+1$, we tested all sums $a + b = n$ and found $s_k(\vx)$ for the vector $x_1 = \cdots = x_b = a$, $x_{b+1} = \cdots = x_n = -b$.
	We let $t$ be the minimum such value $s_k(\vx)$ and executed \BranchAndCut$(n,k,t,\varnothing,\varnothing,\nCk)$, which found no vector with fewer than $t$ nonnegative $k$-sums.
	By executing these algorithms in increasing value of $n$, we have access to $g(n-k,k)$ during execution of  \PropagateNegative($n,k,t,\branchpos,\branchneg,\cutstar$) (Algorithm~\ref{alg:propagation}) when testing $g(n,k) \geq t$ for $n > 2k$.
	Finally, for $n \in \{3k+2,\dots,4k+1\}$, we find $g(n,k) \geq \binom{n-1}{k-1}$ by executing \BranchAndCut$(n,k,\binom{n-1}{k-1},\varnothing,\varnothing,\nCk)$.
\end{proof}

For $k = 7$, \BranchAndCut\ succeeded in computing $g(n,k)$ for $n \leq 23$, but only after resorting to parallel computation.
Previous computations showed an order of magnitude jump in computation time between $n = f(k) - 1$ and $n = f(k)$, so using \BranchAndCut\ to verity $\hat g(24,7) = 0$ seemed intractable.

In the next section, we develop a zero-error randomized propagation algorithm which improves the performance enough for us to compute $f(7)$.

\section{Propagation With Randomness}\label{sec:branchless}

When branching in our branch-and-cut method, we select a set $S$ in $\cutstar$ and make a temporary decision of $\sigma_S(\vx) \leq -1$ or $\sigma_S(\vx) \geq 0$.
Then during our propagation step, the \PropagateNegative\ algorithm (Algorithm~\ref{alg:propagation}) places as many left-most sets into $\genneg$ as possible using Lemmas~\ref{lma:baranyaitrick} and \ref{lma:propagatenegative}.
We now describe a condition that allows us to add sets into $\genpos$ without branching.

Given a set $S$ in $\cutstar$, we can test the linear program for feasibility when $S$ is added to $\genneg$ or $\genpos$.

\begin{observation}\label{lma:propagatepositive}
	If $S$ is a $k$-set where $P(n, k, \genpos, \genneg\cup\{S\})$ is infeasible, then all vectors that are feasible solutions to $P(n, k, \genpos, \genneg)$ are also feasible solutions to $P(n,k,\genpos \cup \{S\}, \genneg)$.
\end{observation}

\begin{observation}\label{lma:propagatenegativeLP}
	If $S$ is a $k$-set where $P(n, k, \genpos \cup \{S\}, \genneg)$ is infeasible, then all vectors that are feasible solutions to $P(n, k, \genpos, \genneg)$ are also feasible solutions to $P(n,k,\genpos, \genneg\cup \{S\})$.
\end{observation}

Given these two lemmas, we could test every set in $\cutpos$ and determine which sets should be added to $\genpos$ or $\genneg$.
The \PropagatePositive\ algorithm (Algorithm~\ref{alg:propagatepos}) places as many right-most sets into $\genpos$ as possible using Lemma~\ref{lma:propagatepositive}.

\begin{algorithm}[tp]
\begin{algorithmic}
\REQUIRE $\cL_k(\genpos),\cR_k(\genneg)$, and $\cutstar$ partition $\nCk$ and $|\cL_k(\genpos)| < t$.
\ENSURE Returns $\genpos_0$, $\genneg_0$, and $\cutstar_0$ such that any $(n,k,t)$-bad vector $\vx$ with $\cutpos_\vx \supseteq \genpos$ and $\cutneg_\vx \supseteq \genneg$ also has $\cutpos_\vx \supseteq \genpos_0$, $\genneg_0\equiv \genneg$, and $\cutstar_0 = \nCk \setminus (\cL_k(\genpos_0)\cup \cR_k(\genneg_0))$.

\vspace{0.5em}

\STATE $\genpos_0, \genneg_0, \cutstar_0 \leftarrow$ \PropagateNegative$(n,k,t,\genpos,\genneg,\cutstar)$

\STATE updated $\leftarrow \True$
\WHILE{updated}
	\STATE update $\leftarrow \False$
\FOR{{\bf all} sets $S \in \cutstar_0$ (in reverse lex order)}
	\IF{$\cP(n,k,\genpos, \genneg\cup \{S\})$ is infeasible}
		\STATE update $\leftarrow \True$	
		\STATE $\genpos_0 \leftarrow \genpos_0\cup \{S\}$.
		\STATE $\cutstar_0 \leftarrow \cutstar_0 \setminus \cL_k(S)$
		\STATE $\genpos_0, \genneg_0, \cutstar_0 \leftarrow$ \PropagateNegative$(n,k,t,\genpos_0,\genneg_0,\cutstar_0)$
	\ENDIF
\ENDFOR
\ENDWHILE

\RETURN $\genpos_0, \genneg_0, \cutstar_0$
\end{algorithmic}
\caption{\label{alg:propagatepos}\PropagatePositive$(n,k,t,\genpos,\genneg,\cutstar)$}
\end{algorithm}

Since \PropagatePositive\ optimizes a linear program for every set in $\cutstar$, this algorithm is too slow to be effective within \BranchAndCut.
However, since \PropagatePositive\ includes a call to \PropagateNegative\ after every addition to $\branchpos$, it is possible that a single call to \PropagatePositive\ results with the linear program $\cP(n,k,\genpos_0,\genneg_0)$ infeasible.
Since \PropagateNegative\ adds negative sets to $\genneg$, there are likely more sets $S \in \cutstar$ where adding $S$ to $\genneg$ leads to an infeasible linear program.
Such sets are added to $\genpos$ by \PropagatePositive, leading to more sets which satisfy the conditions in \PropagateNegative.

Using \PropagatePositive\ in place of \PropagateNegative\ in the propagation step of \BranchAndCut\ is slow for even $k = 5$.
However, for $k = 3$ and $k = 4$, using \PropagatePositive\ terminates in at most three iterations of the loop and requires no branching.
The log of this computation is small enough that we present the computation as proofs that $f(3) \leq 11$ and $f(4) \leq 14$ in Appendices~\ref{apx:k3} and \ref{apx:k4}.

While \PropagatePositive\ is unreasonably slow, \PropagateNegative\ is very fast to test since we have very quick methods for computing $L_k^*(S)$.
If we can find just one set $S$ where adding $S$ to $\genneg$ creates an infeasible linear program, then we can add $S$ to $\branchpos$ and likely the next iteration of \PropagateNegative\ will add more negative sets.

Instead of carefully selecting a set $S$ to peform this test, we simply select a set $S$ in $\cutstar$ uniformly at random.
By randomly selecting sets and testing the linear program, we may quickly find a set that we can guarantee to be in $\cutpos_\vx$ for any $(n,k,t)$-bad vector $\vx \in F_n$.
This is the idea for the algorithm \StochasticPropagation\ (Algorithm~\ref{alg:stochastic}).
Simply, we select a set from $\cutstar$ at random and test if the set fits the hypotheses of Observations~\ref{lma:propagatepositive} or \ref{lma:propagatenegativeLP}.
We continue sampling until either (a) we find such a set and add it to $\genpos$ or $\genneg$, (b) we sample a specified number of sets which all fail these conditions, or (c) we reach a specified time limit.
The sampling limits of (b) and (c) are not listed in Algorithm~\ref{alg:stochastic}, but are both parameters which can be modified in our implementation.

\begin{algorithm}[tp]
	\caption{\label{alg:stochastic}\StochasticPropagation$(n, k, t, \genpos, \genneg, \cutstar)$ --- Randomly test sets for inclusion in $\genneg$ and $\genpos$.}
	\begin{algorithmic}
	\REQUIRE $k \geq 3$, $n \geq k$, $t \leq \binom{n-1}{k-1}$, and $\branchpos, \branchneg \subset \nCk$.
	\ENSURE Sets added to $\genpos$ or $\genneg$ satisfy the hypotheses of Observations~\ref{lma:propagatepositive} or \ref{lma:propagatenegativeLP}.
	
	\LOOP
			\STATE $\genpos, \genneg, \cutstar \leftarrow$ \PropagateNegative$(n, k, t, \genpos, \genneg, \cutstar)$\qquad(\emph{see Algorithm~\ref{alg:propagation}})
		
			\STATE updated $\leftarrow$ False
			\WHILE{\textbf{not} updated}
				\STATE Randomly select a set $S \in \cutstar$
				\IF{$\cP(n, k, \genpos, \genneg \cup \{S\})$ is infeasible}
					\STATE $\genpos \leftarrow \genpos \cup \{S\}$,\quad $\cutstar \leftarrow \cutstar \setminus \cL_k(S)$,\quad updated $\leftarrow$ True
			\IF{$|\cL_k(\genpos)|\geq t$}
				\RETURN $\genpos, \genneg, \cutstar$
			\ENDIF
				\ELSIF{$\cP(n, k, \genpos \cup \{S\}, \genneg)$ is infeasible}
					\STATE $\genneg \leftarrow \genneg \cup \{S\}$,\quad $\cutstar \leftarrow \cutstar \setminus \cR_k(S)$,\quad updated $\leftarrow$ True				
				\ENDIF
			\ENDWHILE
		\ENDLOOP
	
	\RETURN $\genpos, \genneg, \cutstar$
	\end{algorithmic}
\end{algorithm}

\StochasticPropagation\ is a zero-error randomized algorithm: it adds sets to $\genpos$ or $\genneg$ only when previous evidence guarantees that is the correct choice. 
The only effect of the randomness is how many sets actually are determined to be placed in $\genpos$ or $\genneg$.

What is particularly important is that the propagation in \StochasticPropagation\ is stronger than \PropagateNegative, but it is still slower.
For the best performance, we need a balance between the branching procedure and the propagation procedure. 
This balance is found by adjusting the number of random samples between successful updates in \StochasticPropagation, as well as the total time to spend in that algorithm.
Further, we can disable the call to \StochasticPropagation\ until a certain number of $k$-sets have been added to $\branchpos$ or $\branchneg$, thereby strongly constraining the linear program and leading to a more effective random sampling process.

Solving the case $k = 7$ and $n = 22$ required more than 100 days of computation time for the deterministic branch-and-cut method, but required only a few hours using the randomized algorithm.

\subsection{Results Using \StochasticPropagation}\label{sec:results}

Using \BranchAndCut\ with \StochasticPropagation, we computed $g(n,7)$ for $8\leq n \leq 29$ and verified that $f(k) = 3k+2$ for $k =  7$.

\begin{theorem}
	$f(7) = 23$.
\end{theorem}

The computation for this theorem is the same as in Thoerem~\ref{thm:branching}, except we replace \PropagateNegative\ with \StochasticPropagation.

%
%
%

\subsection{Implementation}

The \BranchAndCut\ algorithm was implemented in the \textsl{MMSConjecture} project within the \textsl{SearchLib} collection\footnote{\textsl{SearchLib} is available at \url{http://www.math.illinois.edu/~stolee/SearchLib/}.}.
The software is available for download, reuse, and modification under the GNU Public License 3.0.

For $k\in \{4,5\}$, we were able to verify the statements using the exact arithmetic solver supplied with GLPK~\cite{GLPK}, but for $k \geq 6$ this method was too slow and instead was verified using the floating-point linear programming software CPLEX~\cite{CPLEX}.
The number of search nodes for each search, as well as the amount of computation time for each case are presented in Tables~\ref{tab:bandctiming34}, \ref{tab:bandctiming56}, and \ref{tab:k7}.
Empty cells refer to computations that were too long to complete (but only for cases where $k$ divides $n$) and cells containing ``---'' refer to experiments reporting less than 0.01 seconds of computation time.
In addition to the data presented here, all collected statistics and sharp examples are available online\footnote{All data is available at \url{http://www.math.illinois.edu/~stolee/data.htm}.}.

Execution times under one day were executed in a single process on a 2.3 GHz Intel Core i5 processor.
Longer execution times are from parallel execution on the Open Science Grid \cite{OpenScienceGrid} using the University of Nebraska Campus Grid \cite{WeitzelMS}.
The nodes available on the University of Nebraska Campus Grid consist of Xeon and Opteron processors with a range of speed between 2.0 and 2.8 GHz.

\section{Sharp Examples, Uniqueness, and the Strengthened Conjecture} 
\label{sec:strong}

In Tables \ref{tab:bandctiming34}, \ref{tab:bandctiming56}, and \ref{tab:k7}, the right-most column contains descriptions of the vectors $\vx$ with $s_k(\vx)$ of minimum value.
For $n \geq f(k)$, the vectors listed had $s_k(\vx)$ of minimum value while also having $\sigma_T(\vx) < 0$ for $T = \{1, n-k+2, \dots, n\}$.
Observe that for all $n < f(k)$, the extremal examples use only two numbers, and have the form $a^b\ (-b)^a$ where $a + b = n$.

Recall Conjecture~\ref{conj:ours}, where we claim $f(k)$ is exactly equal to $N_k$, where $N_k$ is the minimum integer such that $\binom{N_k-3}{k} \geq \binom{N_k-1}{k-1}$.
We make this conjecture for two reasons.
First, the example of $3^{n-3}\ (-(n-3))^3$ was known to have fewer than $\binom{n-1}{k-1}$ nonnegative $k$-sums when $3k < n < N_k$ by previous authors (for example, see~\cite{C} where these vectors were used to show a lower bound $f(k) \geq \frac{22k}{7}$), but no examples were previously discovered for $n \geq N_k$.
Second, our search for extremal vectors found no violation to this bound, but also showed that the extremal examples have the form $a^b\ (-b)^a$ when $n \leq N_k$ and $k$ does not divide $n$.
We therefore tested all vectors of the form $a^b\ (-b)^a$ for all $k \leq 250$.
For all numbers $n$ with $3k < n < N_k$ the example $3^{n-3}\ (-(n-3))^3$ was the best vector of this form, and for $N_k \leq n < 4k$ the best vector was $(n-1)^1\ (-1)^{n-1}$.

Figure~\ref{fig:plot} contains a plot of $N_k/k$ and the line $y = \lim_{k\to\infty}N_k/k$, to demonstrate the conjectured values of $f(k)/k$.

\begin{figure}[tp]
	\centering
	\includegraphics[height=2.5in]{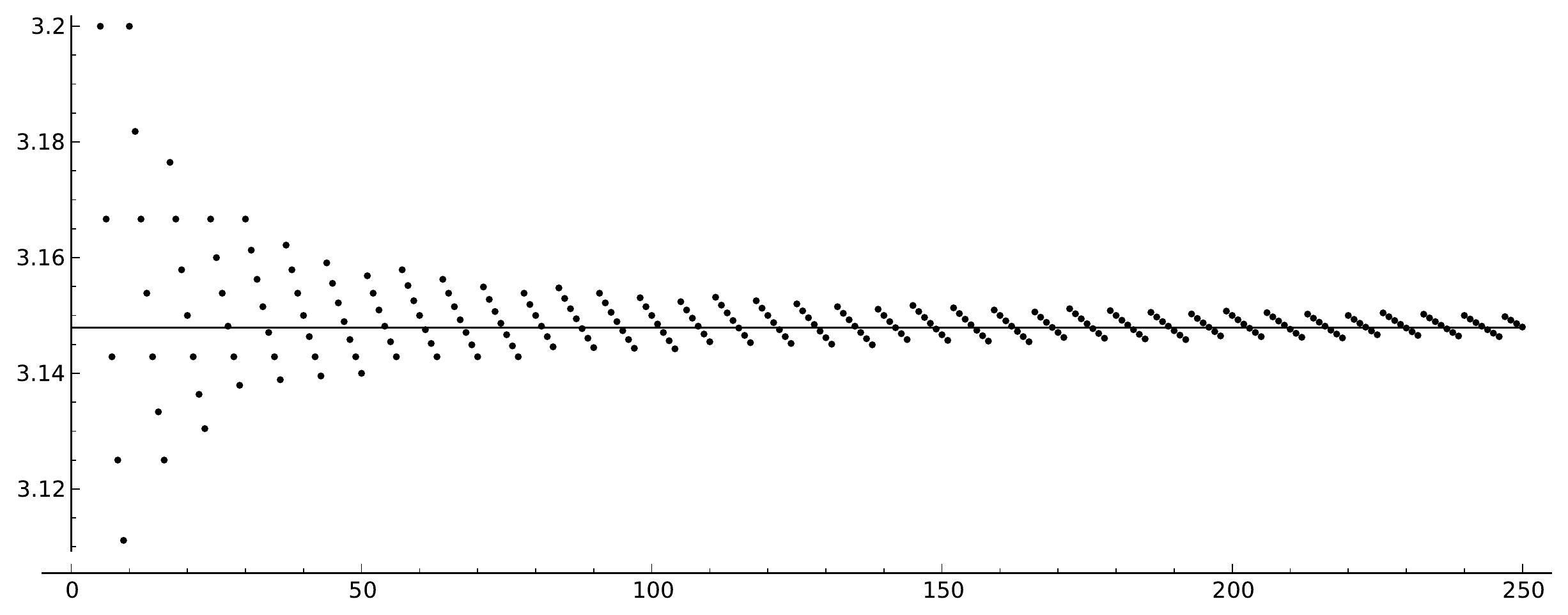}
	\caption{\label{fig:plot}Values of $N_k/k$ for $5 \leq k \leq 250$ and the line $y = \lim_{k\to\infty}  N_k/k$.}
\end{figure}

Another strengthening of the conjecture that follows from our method (for $k \leq \maxkvalue$) and that of Chowdhury (for $k = 3$) is that when $n \geq f(k)$, a vector $\vx \in F_n$ has $s_k(\vx) = \binom{n-1}{k-1}$ only if $x_1 + x_{n-k+2} + \cdots + x_n$ is nonnegative.
Essentially, any sharpness example contains the same set of nonnegative $k$-sums as the sharpness example $(n-1)^1\ (-1)^{n-1}$.

We test this example is essentially unique by searching for an vector $\vx \in F_n$ where $\sigma_T(\vx) < 0$ for $T = \{1,n-k+2,\dots,n\}$ and $s_k(\vx) \leq \binom{n-1}{k-1}$.
If no such vector is found, then the sharpness example is ``unique,'' since all vectors $\vx \in F_n$ with $s_k(\vx) \leq \binom{n-1}{k-1}$ have $\cutpos_\vx = \cL_k(T)$.
Define $g_s(n,k)$ to be the minimum $s_k(\vx)$ over all $\vx \in F_n$ where $\sigma_T(\vx) < 0$.
We can use our target value $t$ to find an $(n,k,t)$-bad vector $\vx \in F_n$ with $\sigma_T(\vx) < 0$.
While the bound $t \leq \binom{n-1}{k-1}$ in the hypothesis of Lemma~\ref{lma:baranyaitrick} does not hold, observe that since $\sigma_T(\vx) < 0$ the conclusion of the lemma does hold in this scenario.
Therefore, \PropagateNegative\ (Algorithm~\ref{alg:propagation}) remains correct when  verifying $g_s(n,k) \geq t$ using a call to \BranchAndCut$(n, k, t, \varnothing, \{ T\}, \nCk\setminus \cR_k(T))$.

Values of $g_s(n,k)$ were computed for $3 \leq k\leq 6$, $f(k)\leq n < f(k) + k$ and are given in Table~\ref{tab:strongvalues}.
The sharp examples for vectors $\vx\in F_n$ with $\sigma_T(\vx) < 0$ are given in the final columns of Tables~\ref{tab:bandctiming34} and \ref{tab:bandctiming56}.
Observe that the sharp examples make a ``phase transition'' at $n = 4k$, where the sharp examples have the form $a^b\ (-b)^a$ for all $n < 4k$, but then the examples with $n \geq 4k$ have at least three distinct values.
This may be hinting to a deeper truth concerning the originally conjectured bound of $f(k) \leq 4k$.

\subsection{Conclusions}

We developed two methods for verifying $g(n,k)\geq t$, which we used to prove our strengthening of the Manickam-Mikl\'os-Singhi conjecture for all $k \leq 7$, extending the previously best result of $k \leq 3$~\cite{C}.
Our branch-and-cut method \BranchAndCut\ uses a branching procedure which we prove will, in finite time, verify $g(n,k)\geq t$ or find a vector $\vx \in F_n$ with $s_k(\vx) < t$.
Using the randomized propagation algorithm \StochasticPropagation, we computed $f(k)$ and the values of $g(n,k)$ for all $k \leq \maxkvalue$ and $k < n < f(k) + k$.
Our implementations were not successful in extending our results to larger values of $k$ due to a combination of large computation time and memory requirements.

\section*{Acknowledgments}

The authors thank Ameera N. Chowdhury for many interesting discussions on this problem.
The authors also thank Igor Pak for suggesting a computational approach which led to our branch-and-cut method.

\appendix
\clearpage
\small

\section{Computation Data}

\def\tablefontsize{\footnotesize}

\begin{table}[H]
	\centering
	\footnotesize
	\begin{tabular}[h]{c|r|r|r|r|r@{.}l|r@{.}l|c}
		 $k$ & \multicolumn{1}{c|}{$n$} &\multicolumn{1}{c|}{$g(n,k)$} & \multicolumn{1}{c|}{$\hat g(n,k)$} &  \multicolumn{1}{c|}{Nodes} & \multicolumn{2}{c|}{GLPK}& \multicolumn{2}{c|}{CPLEX} & Strong Example\\
		\hline\hline
	 3 &       4 &        1 &     2   &       2 & \multicolumn{2}{c|}{---} &  \multicolumn{2}{c|}{---} & 
	 	$1^3\ (-3)^1$ \\\hline
	 3 &       5 &        3 &      3  &       2 & \multicolumn{2}{c|}{---} &  \multicolumn{2}{c|}{---} & 
	 	$3^2\ (-2)^3$ \\\hline
	 3 &       6 &       10 &      0 &       8 & \multicolumn{2}{c|}{---} &  \multicolumn{2}{c|}{---} & 
	 
	 \\\hline
	 3 &       7 &      10 &        5 &       2 & \multicolumn{2}{c|}{---} &  \multicolumn{2}{c|}{---} &
	 $2^5\ (-5)^2$
	  \\\hline
	 3 &       8 &      16 &       5 &       2 &  \multicolumn{2}{c|}{---} &  \multicolumn{2}{c|}{---} &
	 $5^3\ (-3)^5$
	 \\\hline
	 3 &       9 &      28 &       0 &       2 & \multicolumn{2}{c|}{---} &  \multicolumn{2}{c|}{---} &
	 \\\hline
	 3 &      10 &      35 &       1 &       4 &  \multicolumn{2}{c|}{---} &  \multicolumn{2}{c|}{---} &
	 $3^7\ (-7)^3$
	 \\\hline
	 3 &      11 &      45 &       0 &       6 &  \multicolumn{2}{c|}{---} &  \multicolumn{2}{c|}{---} &
	 $7^4\ (-4)^7$
	  \\\hline
	 3 &      12 &      55 &       0 &       2 &  \multicolumn{2}{c|}{---} &  \multicolumn{2}{c|}{---} &
	  $7^5\ (-5)^7$
	 \\\hline
	 3 &      13 &      66 &       0 &       2 &  \multicolumn{2}{c|}{---} &  \multicolumn{2}{c|}{---} &
	  $4^9\ (-9)^4$
	 \\\hline\hline

	 4 &       5 &       1 &       3 &       2 & \multicolumn{2}{c|}{---} &  \multicolumn{2}{c|}{---} & 
	 	$1^4\ (-4)^1$ \\\hline
	 4 &       6 &       5 &       5 &       2 & \multicolumn{2}{c|}{---} &  \multicolumn{2}{c|}{---} & 
	 	$1^5\ (- 5)^1$
	 \\\hline
	 4 &       7 &      10 &      10 &       2 & \multicolumn{2}{c|}{---} &  \multicolumn{2}{c|}{---} &
	 $5^2\ (-2)^5$
	  \\\hline
	 4 &       8 &      35 &       0 &     268 &                   0&47 s &                   0&31 s &
	 \\\hline
	 4 &       9 &      35 &      21 &       4 & \multicolumn{2}{c|}{---} &  \multicolumn{2}{c|}{---} &
	 $2^7\ (- 7)^2$
	 \\\hline
	 4 &      10 &      70 &      14 &      28 &                   0&11 s &                   0&03 s &
	 $2^8\ (- 8)^2$
	 \\\hline
	 4 &      11 &      92 &      28 &      10 &                   0&07 s &                   0&01 s &
	 $8^3\ (-3)^8$
	  \\\hline
	 4 &      12 &     165 &       0 &      50 &                   0&35 s &                   0&11 s &
	 \\\hline
	 4 &      13 &     210 &      10 &      52 &                   0&82 s &                   0&15 s &
	$3^{10}\ (-10)^3$
	 \\\hline
	 4 &      14 &     286 &       0 &      30 &                   0&65 s &                   0&10 s &
	  $10^4\ (-4)^{10}$
	 \\\hline
	 4 &      15 &     364 &       0 &      24 &                   0&51 s &                   0&10 s &
	 $11^4\ (-4)^{11}$
	 \\\hline
	 4 &      16 &     455 &       0 &       6 &                   0&15 s &                   0&04 s &
	  $35^1\ 3^{10}\ (-16)^5$
	 \\\hline
	 4 &      17 &     560 &       0 &       8 &                   0&31 s &                   0&07 s &
	 $38^1\ 4^{10}\ (-13)^6$
	 \\\hline\hline
	 5 &       6 &       1 &       4 &       2 & \multicolumn{2}{c|}{---} &  \multicolumn{2}{c|}{---} & $1^5\ (-5)^1$ \\\hline
	 5 &       7 &       6 &       9 &       2 & \multicolumn{2}{c|}{---} &  \multicolumn{2}{c|}{---} & $1^6\ (-6)^1$ \\\hline
	 5 &       8 &      16 &      19 &       4 & \multicolumn{2}{c|}{---} &  \multicolumn{2}{c|}{---} & $3^5\ (-5)^3$ \\\hline
	 5 &       9 &      35 &      35 &       4 &                   0&01 s &  \multicolumn{2}{c|}{---} & $7^2\ (-2)^7$ \\\hline
	 5 &      10 &     126 &       0 &         &  \multicolumn{2}{c|}{}   & \multicolumn{2}{c|}{}     &  \\\hline
	 5 &      11 &     126 &      84 &      10 &                   0&10 s &                   0&03 s  & $2^9\ (-9)^2$ \\\hline
	 5 &      12 &     246 &      84 &      92 &                   2&61 s &                   0&26 s  & $7^5\ (-5)^7$  \\\hline
	 5 &      13 &     405 &      90 &     234 &                  14&22 s &                   1&08 s  & $10^3\ (-3)^{10}$\\\hline
	 5 &      14 &     550 &     165 &      44 &                   3&37 s &                   0&44 s  & $11^3\ (-3)^{11}$\\\hline
	 5 &      15 &    1001 &       0 &     996 & \multicolumn{2}{c|}{}    &                  12&03 s  & \\\hline
	 5 &      16 &    1287 &      78 &     342 &                   2&43 m &                   7&78 s  & $3^{13}\ (-3)^{13}$ \\\hline
	 5 &      17 &    1820 &       0 &     702 &                   5&48 m &                  23&26 s  
	& $10^7\ (-7)^{10}$
	  \\\hline
	 5 &      18 &    2380 &       0 &     364 &                   4&46 m &                  20&77 s  
	 & $14^4\ (-4)^{14}$
	 \\\hline
	 5 &      19 &    3060 &       0 &     192 &                   3&64 m &                  16&54 s  
	& $15^4\ (-4)^{15}$
	 \\\hline
	 5 &      20 &    3876 &       0 &      64 &                   2&13 m &                   9&29 s  
	 & $79^1\ 19^1\ (-1)^{14}\ (-21)^4$
	 \\\hline
	 5 &      21 &    4845 &       0 &      64 &                   3&09 m &                  13&53 s  
	 & $67^1\ 4^{13}\ (-17)^{7}$
	 \\\hline
	 \hline
	\end{tabular}
	
	\caption{\label{tab:bandctiming34}Data for branch-and-cut method with GLPK and CPLEX, for $k \in \{3,4,5\}$.}
\end{table}

\begin{table}[htp]
	\centering
	\footnotesize
	\begin{tabular}[h]{c|r|r|r|r|r@{.}l|r@{.}l|r@{.}l|c}
		 $k$ & \multicolumn{1}{c|}{$n$} &\multicolumn{1}{c|}{$g(n,k)$} & \multicolumn{1}{c|}{$\hat g(n,k)$} &  \multicolumn{1}{c|}{Nodes} & \multicolumn{2}{c|}{GLPK}& \multicolumn{2}{c|}{CPLEX} & \multicolumn{2}{c|}{Stochastic} & Strong Example\\
	\hline
	 \hline
	 6 &       7 &       1 &       5 &       2 & \multicolumn{2}{c|}{---} &  \multicolumn{2}{c|}{---} & \multicolumn{2}{c|}{---} & $1^6\ (-6)^1$ \\\hline
	 6 &       8 &       7 &      14 &       2 & \multicolumn{2}{c|}{---} &  \multicolumn{2}{c|}{---} & \multicolumn{2}{c|}{---} & $1^7\ (-7)^1$ \\\hline
	 6 &       9 &      28 &      28 &       6 &                   0&01 s &  \multicolumn{2}{c|}{---} & 	 0&03 s & $1^8\ (-8)^1$ \\\hline
	 6 &      10 &      70 &      56 &      32 &                   0&20 s &                   0&05 s  &		0&25 s & $8^2\ (-2)^8$ \\\hline
	 6 &      11 &     126 &     126 &      10 &                   0&09 s &                   0&03 s  &		0&05 s & $9^2\ (-2)^9$\\\hline
	 6 &      12 &     462 &       0 &         &  \multicolumn{2}{c|}{}   & \multicolumn{2}{c|}{}     & \multicolumn{2}{c|}{   } \\\hline
	 6 &      13 &     462 &     330 &      24 &                  1&88 s &     0&16 s  & 0&21 s & $2^{11}\ (-11)^2$ \\\hline
	 6 &      14 &     924 &     363 &     294 &                   1&32 m &                   6&17 s  &    8&41 s & $2^{12}\ (-12)^2$ \\\hline
	 6 &      15 &    1705 &     297 &   12408 &                   2&94 h &                   6&47 m  &    4&76 m & $12^3\ (-3)^{12}$ \\\hline
	 6 &      16 &    2431 &     572 &    1296 &                  34&19 m &                   1&68 m  &   54&20 s & $13^3\ (-3)^{13}$\\\hline
	 6 &      17 &    3367 &    1001 &     266 &                  10&84 m &                  44&14 s  &   17&04 s & $14^3\ (-3)^{14}$\\\hline
	 6 &      18 &    6188 &       0 &  183960 & \multicolumn{2}{c|}{}    &                   9&66 h  & \multicolumn{2}{c|}{}\\\hline
	 6 &      19 &    8008 &     560 &    7262 & \multicolumn{2}{c|}{}    &                  47&46 m  & 13&74 m & $3^{16}\ (-16)^{3}$\\\hline
	 6 &      20 &   11628 &       0 &   27696 & \multicolumn{2}{c|}{}    &                   4&58 h  &  2&32 h
	& $3^{17}\ (-17)^{3}$ \\\hline
	 6 &      21 &   15504 &       0 &    8932 & \multicolumn{2}{c|}{}    &                   2&48 h  &  1&11 h
	& $17^4\ (-4)^{17}$ \\\hline
	 6 &      22 &   20349 &       0 &    4622 & \multicolumn{2}{c|}{}    &                   2&06 h  &  59&81 m
	& $18^4\ (-4)^{18}$  \\\hline
	 6 &      23 &   26334 &       0 &    2378 & \multicolumn{2}{c|}{}    &                   1&66 h  &  38&51 m
	& $19^4\ (-4)^{19}$ \\\hline
	 6 &      24 &   33649 &       0 &     764 & \multicolumn{2}{c|}{}    &                  49&35 m  & \multicolumn{2}{c|}{}
	 & $33^1\ 1^{16}\ (-7)^7$
	 \\\hline
	 6 &      25 &   42504 &       0 &     744 & \multicolumn{2}{c|}{}    &                   1&15 h  & 26&62 m 
	 & $104^1\ 4^{16}\ (-21)^{8}$ \\\hline
	\end{tabular}
		
	\caption{\label{tab:bandctiming56}Data for the branch-and-cut method using GLPK and CPLEX, for $k = 6$.}
\end{table}

\begin{table}[htp]
	\centering
	\footnotesize
	\begin{tabular}[h]{c|r|r|r|r|r@{.}l|r@{.}l|c}
		 $k$ & \multicolumn{1}{c|}{$n$} &\multicolumn{1}{c|}{$g(n,k)$} & \multicolumn{1}{c|}{$\hat g(n,k)$} &  \multicolumn{1}{c|}{Nodes} & \multicolumn{2}{c|}{Deterministic} & \multicolumn{2}{c|}{Stochastic} & Sharp Example \\
		\hline\hline
	 7 &       8 &       1 &        6 &       2  &  \multicolumn{2}{c|}{---}  &  \multicolumn{2}{c|}{---} & $1^7\ (-7)^1$ \\\hline
	 7 &       9 &       8 &       20 &       2  &  \multicolumn{2}{c|}{---}  &  \multicolumn{2}{c|}{---} & $1^8\ (-8)^1$ \\\hline
	 7 &      10 &      35 &       48 &       8  &  \multicolumn{2}{c|}{---}  &  \multicolumn{2}{c|}{---} & $7^3\ (-3)^7$ \\\hline
	 7 &      11 &      92 &      118 &      12  &  0&02 s   &  0&02 s & $3^8\ (-8)^3$ \\\hline
	 7 &      12 &     246 &      216 &     100  &  0&39 s   &  0&92 s & $5^7\ (-7)^5$ \\\hline
	 7 &      13 &     462 &      462 &          &  0&36 s   &  0&13 s & $11^2\ (-2)^{11}$ \\\hline
	 7 &      14 &    1716 &        0 &      26  & \multicolumn{2}{c|}{}   &  \multicolumn{2}{c|}{}&  \\\hline
	 7 &      15 &    1716 &    1287 &      58  &\ 6&02 s  &  9&59 s
	 	& $2^{13}\ (-13)^2$ \\\hline
	 7 &      16 &     3432 &    1573 &     1562  & 4&97 m   &  15&40 m
	 	& $2^{14}\ (-14)^2$ \\\hline
	 7 &      17 &     6116 &    1892 &    26852 &  2&61 h   &  1&07 h
	 	& $12^{5}\ (-5)^{12}$ \\\hline
	 7 &      18 &    10296 &   2080   & 450772 &  3&44 d       &  1&02 h
	 	& $5^{13}\ (-13)^{5}$ \\\hline
	 7 &      19 &    14924 &   3640  &   28778  &  1&25 d   &  1&15 h
	 	& $16^3\ (-3)^{16}$ \\\hline
	 7 &      20 &    20944 &   6188  &    3615 &  8&51 h   &  28&80 m
	 	& $17^3\ (-3)^{17}$ \\\hline
	 7 &      21 &   38760 &       0 &          &  \multicolumn{2}{c|}{}     &  \multicolumn{2}{c|}{}  
	 	&  \\\hline
	 7 &      22 &   50388 &  3876  &    795236 & \ 160&57 d     &  3&08 h
	 	& $3^{19}\ (-19)^{3}$ \\\hline
	 7 &      23 &  74613 &  0  &     13013$^*$     &  \multicolumn{2}{c|}{}    &   8&09 d$^*$
		 	&      \\\hline
	 7 &      24 & 100947  &  0  &    7870$^*$      &  \multicolumn{2}{c|}{}    &  4&61 d$^*$
		 	&      \\\hline
	 7 &      25 & 134596  &  0  &    6531$^*$   &  \multicolumn{2}{c|}{}    &  3&85 d$^*$
		 	&      \\\hline
	 7 &      26 & 177100  &  0  &    12718$^*$      &  \multicolumn{2}{c|}{}    &  9&25 d$^*$
		 	&      \\\hline
	 7 &      27 & 230230  &  0  &    30807$^*$      &  \multicolumn{2}{c|}{}    &  19&18 d$^*$
		 	&      \\\hline
	 7 &      28 & 296010  &  0  &    5564$^*$ &  \multicolumn{2}{c|}{}    &  2&51 d$^*$
		 	&      \\\hline
	 7 &      29 & 376740  &  0  &    6002$^*$      &  \multicolumn{2}{c|}{}    &  3&38 d$^*$   
		 	&      \\\hline
\end{tabular}
\vspace{1em}

\footnotesize
$^*$ These node counts and CPU times are averages of at least three runs using \StochasticPropagation.

\caption{\label{tab:k7}Completed computations for $k = 7$ using CPLEX.}
\end{table}

\clearpage

\begin{table}[htp]
	\centering
	\footnotesize
	\begin{tabular}[h]{r||c|c|c||c|c|c|c||c|c|c|c|c}
		$k$ & 3 & 3 & 3 & 4 & 4 & 4& 4 & 5 & 5& 5& 5& 5 \\
		\hline
		$n$ & 11 & 12 & 13 & 14 & 15 & 16 & 17 & 17 & 18 & 19 & 20 & 21  \\
		\hline
		$g(n,k)$ & 45 & 55 & 66 & 286 & 364 & 455 & 560 & 1820 & 2380 & 3060 & 3876 & 4845 \\
		\hline
		$g_s(n,k)$ & 46 & 80 & 84 & 311 & 375 & 455 & 750 & 1946 & 2562 & 3165 & 4876 & 6097
	\end{tabular}
	
	\vspace{0.5em}
	\begin{tabular}[h]{r||c|c|c|c|c|c}
		$k$ & 6 & 6 & 6 & 6 & 6 & 6\\
		\hline
		$n$ & 20 & 21 & 22 & 23 & 24 & 25 \\
		\hline
		$g(n,k)$ & 11628 & 15504 & 20349 & 26334 & 33649 & 42054 \\
		\hline
		$g_s(n,k)$ & 12376 & 17136 &  21777 &  27303 & 39836 &   50456
		\\
		\hline
		 Time & 5.71 d & 15.91 d & 2.26 d & 19.70 h & 67.60 d &  30.00 d  
	\end{tabular}
	\caption{\label{tab:strongvalues}Comparisons of $g(n,k)$ and $g_s(n,k)$ when $f(k) \leq n < f(k) + k$.}
\end{table}

\section{A Computer-Generated Proof that $f(3) \leq 11$}\label{apx:k3}

The following proofs were created by executing \BranchAndCut\ with propagation step \PropagatePositive\  (Algorithm~\ref{alg:propagatepos}) and writing down the $k$-sums which are determined to be strictly negative or nonnegative.

\begin{theorem}
	$g(3,11) = \binom{10}{2} = 45$.
\end{theorem}

\begin{proof}
The following sums generated by \PropagateNegative\ (Algorithm~\ref{alg:propagation})  must be strictly negative or we have at least 45 of nonnegative sets:
\begin{align*}
x_{ 1} +  & x_{ 6} + x_{11}	&	x_{ 1} +  & x_{ 8} + x_{10}	&	x_{ 2} +  & x_{ 5} + x_{11}\\
	x_{ 2} +  & x_{ 6} + x_{10}	&	x_{ 2} +  & x_{ 7} + x_{ 9}	&	x_{ 3} +  & x_{ 4} + x_{11}\\
	x_{ 3} +  & x_{ 5} + x_{ 9}	&	x_{ 3} +  & x_{ 7} + x_{ 8}	&	x_{ 4} +  & x_{ 6} + x_{ 8}
	\end{align*}
The following sums generated by \PropagatePositive\ (Algorithm~\ref{alg:propagatepos}) must be nonnegative or else the associated linear program (Definition~\ref{def:lp}) becomes infeasible:
\begin{align*}
	x_{ 4} + &x_{ 6} + x_{ 7}	&	x_{ 4} + &x_{ 5} + x_{ 8}	&	x_{ 3} + &x_{ 4} + x_{10}
\end{align*}
These positive sets now force at least 56 nonnegative $3$-sums, and our target was $45$ nonnegative $3$-sums.
\end{proof}

%
%

\begin{theorem}
	$g(3,13) = \binom{12}{2} = 66$.
\end{theorem}

\begin{proof}
The following sums generated by  \PropagateNegative\ (Algorithm~\ref{alg:propagation}) must be strictly negative or we have at least 66 nonnegative sets:
\begin{align*}
x_{ 1} +  & x_{ 5} + x_{13}	&	x_{ 1} +  & x_{ 6} + x_{12}	&	x_{ 1} +  & x_{ 7} + x_{11}\\
	x_{ 3} +  & x_{ 5} + x_{12}	&	x_{ 3} +  & x_{ 6} + x_{10}	&	x_{ 4} +  & x_{ 5} + x_{11}\\
	x_{ 4} +  & x_{ 7} + x_{ 9}	&	
\end{align*}
	The associated linear program is infeasible.		
\end{proof}

\section{A Computer-Generated Proof that $f(4) \leq 14$}\label{apx:k4}

The following proofs were created by executing \BranchAndCut\ with propagation step \PropagatePositive\ (Algorithm~\ref{alg:propagatepos}) and writing down the $k$-sums which are determined to be strictly negative or nonnegative.

\begin{theorem}
	$g(4,14) = \binom{13}{3} = 286$.
\end{theorem}

\begin{proof}
The following sums generated by \PropagateNegative\ (Algorithm~\ref{alg:propagation})  must be strictly negative or we have at least 286 nonnegative sets:
\begin{align*}
	x_{ 1} + x_{ 7} +  & x_{13} + x_{14}	&	x_{ 1} + x_{ 8} +  & x_{12} + x_{14}	&	x_{ 1} + x_{ 9} +  & x_{11} + x_{14}\\
	x_{ 2} + x_{ 5} +  & x_{10} + x_{14}	&	x_{ 2} + x_{ 6} +  & x_{ 9} + x_{14}	&	x_{ 2} + x_{ 6} +  & x_{10} + x_{13}\\
	x_{ 2} + x_{ 8} +  & x_{ 9} + x_{13}	&	x_{ 3} + x_{ 4} +  & x_{11} + x_{14}	&	x_{ 3} + x_{ 5} +  & x_{ 9} + x_{14}\\
	x_{ 3} + x_{ 5} +  & x_{10} + x_{13}	&	x_{ 3} + x_{ 6} +  & x_{ 8} + x_{14}	&	x_{ 3} + x_{ 6} +  & x_{ 9} + x_{13}\\
	x_{ 3} + x_{ 6} +  & x_{10} + x_{12}	&	x_{ 3} + x_{ 7} +  & x_{ 8} + x_{13}	&	x_{ 3} + x_{ 7} +  & x_{ 9} + x_{12}\\
	x_{ 4} + x_{ 5} +  & x_{ 8} + x_{14}	&	x_{ 4} + x_{ 5} +  & x_{ 9} + x_{13}	&	x_{ 4} + x_{ 6} +  & x_{ 8} + x_{13}\\
	x_{ 4} + x_{ 6} +  & x_{ 9} + x_{12}	&	x_{ 4} + x_{ 8} +  & x_{10} + x_{11}	&	x_{ 5} + x_{ 7} +  & x_{ 8} + x_{12}\\
	x_{ 5} + x_{ 7} +  & x_{10} + x_{11}	&	x_{ 6} + x_{ 8} +  & x_{ 9} + x_{11}	&	
\end{align*}

The following sums generated by \PropagatePositive\ (Algorithm~\ref{alg:propagatepos}) must be nonnegative or else the associated linear program becomes infeasible:
\begin{align*}
	x_{ 3} + x_{ 4} + &x_{ 7} + x_{ 9}	&	x_{ 3} + x_{ 4} + &x_{ 6} + x_{10}	&	x_{ 2} + x_{ 4} + &x_{ 8} + x_{10}\\
	x_{ 1} + x_{ 7} + &x_{ 8} + x_{ 9}	&	x_{ 1} + x_{ 6} + &x_{ 8} + x_{11}	&	x_{ 1} + x_{ 5} + &x_{ 9} + x_{11}\\
	x_{ 1} + x_{ 5} + &x_{ 7} + x_{12}	&	x_{ 1} + x_{ 4} + &x_{10} + x_{11}	&	x_{ 1} + x_{ 4} + &x_{ 9} + x_{12}\\
	x_{ 1} + x_{ 4} + &x_{ 8} + x_{13}	&	x_{ 1} + x_{ 4} + &x_{ 5} + x_{14}	&	x_{ 1} + x_{ 2} + &x_{10} + x_{13}\\
	x_{ 1} + x_{ 2} + &x_{ 8} + x_{14}	&	
\end{align*}
These positive sets now force at least 199 nonnegative $4$-sums.

The following sums generated by \PropagateNegative\ (Algorithm~\ref{alg:propagation}) must be strictly negative or we have  at least 286 of nonnegative sets:
\begin{align*}
	x_{ 1} + x_{ 7} +  & x_{12} + x_{14}	&	x_{ 1} + x_{ 8} +  & x_{11} + x_{14}	&	x_{ 2} + x_{ 4} +  & x_{11} + x_{14}\\
	x_{ 2} + x_{ 5} +  & x_{ 9} + x_{14}	&	x_{ 2} + x_{ 5} +  & x_{11} + x_{13}	&	x_{ 2} + x_{ 6} +  & x_{ 8} + x_{14}\\
	x_{ 2} + x_{ 6} +  & x_{ 9} + x_{13}	&	x_{ 2} + x_{ 7} +  & x_{10} + x_{12}	&	x_{ 3} + x_{ 4} +  & x_{10} + x_{14}\\
	x_{ 3} + x_{ 4} +  & x_{12} + x_{13}	&	x_{ 3} + x_{ 5} +  & x_{ 8} + x_{14}	&	x_{ 3} + x_{ 5} +  & x_{ 9} + x_{13}\\
	x_{ 3} + x_{ 5} +  & x_{10} + x_{12}	&	x_{ 3} + x_{ 6} +  & x_{ 7} + x_{14}	&	x_{ 3} + x_{ 6} +  & x_{ 8} + x_{13}\\
	x_{ 3} + x_{ 6} +  & x_{ 9} + x_{12}	&	x_{ 3} + x_{ 7} +  & x_{ 8} + x_{12}	&	x_{ 3} + x_{ 7} +  & x_{10} + x_{11}\\
	x_{ 3} + x_{ 8} +  & x_{ 9} + x_{11}	&	x_{ 4} + x_{ 5} +  & x_{ 7} + x_{14}	&	x_{ 4} + x_{ 5} +  & x_{ 8} + x_{13}\\
	x_{ 4} + x_{ 5} +  & x_{ 9} + x_{12}	&	x_{ 4} + x_{ 6} +  & x_{ 7} + x_{13}	&	x_{ 4} + x_{ 6} +  & x_{ 8} + x_{12}\\
	x_{ 4} + x_{ 6} +  & x_{ 9} + x_{11}	&	x_{ 5} + x_{ 7} +  & x_{ 8} + x_{11}	&	
\end{align*}

The following sums generated by \PropagatePositive\ (Algorithm~\ref{alg:propagatepos}) must be nonnegative or else the associated linear program becomes infeasible:
\begin{align*}
	x_{ 3} + x_{ 5} + &x_{ 6} + x_{ 9}	&	x_{ 3} + x_{ 4} + &x_{ 8} + x_{10}	&	x_{ 2} + x_{ 6} + &x_{ 7} + x_{10}\\
	x_{ 2} + x_{ 5} + &x_{ 8} + x_{10}	&	x_{ 2} + x_{ 4} + &x_{ 9} + x_{10}	&	x_{ 1} + x_{ 8} + &x_{ 9} + x_{10}\\
	x_{ 1} + x_{ 7} + &x_{10} + x_{11}	&	x_{ 1} + x_{ 7} + &x_{ 8} + x_{13}	&	x_{ 1} + x_{ 6} + &x_{ 9} + x_{12}\\
	x_{ 1} + x_{ 5} + &x_{10} + x_{13}	&	x_{ 1} + x_{ 5} + &x_{ 8} + x_{14}	&	x_{ 1} + x_{ 4} + &x_{10} + x_{14}
\end{align*}
These positive sets now force at least 265 nonnegative $4$-sums.

The following sums generated by \PropagateNegative\ (Algorithm~\ref{alg:propagation}) must be strictly negative or we have  at least 286 nonnegative sets:
\begin{align*}
	x_{ 1} + x_{ 5} +  & x_{12} + x_{14}	&	x_{ 1} + x_{ 6} +  & x_{11} + x_{14}	&	x_{ 1} + x_{ 7} +  & x_{12} + x_{13}\\
	x_{ 1} + x_{ 8} +  & x_{10} + x_{14}	&	x_{ 1} + x_{ 8} +  & x_{11} + x_{13}	&	x_{ 2} + x_{ 3} +  & x_{ 9} + x_{14}\\
	x_{ 2} + x_{ 3} +  & x_{10} + x_{13}	&	x_{ 2} + x_{ 4} +  & x_{ 7} + x_{13}	&	x_{ 2} + x_{ 4} +  & x_{ 9} + x_{12}\\
	x_{ 2} + x_{ 5} +  & x_{ 6} + x_{14}	&	x_{ 2} + x_{ 5} +  & x_{ 8} + x_{12}	&	x_{ 2} + x_{ 6} +  & x_{ 9} + x_{11}\\
	x_{ 3} + x_{ 4} +  & x_{ 6} + x_{13}	&	x_{ 3} + x_{ 4} +  & x_{ 8} + x_{12}	&	x_{ 3} + x_{ 5} +  & x_{ 7} + x_{12}\\
	x_{ 3} + x_{ 5} +  & x_{ 8} + x_{11}	&	x_{ 3} + x_{ 6} +  & x_{ 7} + x_{11}	&	x_{ 3} + x_{ 7} +  & x_{ 9} + x_{10}\\
	x_{ 4} + x_{ 5} +  & x_{ 6} + x_{12}	&	x_{ 4} + x_{ 5} +  & x_{ 7} + x_{11}	&	x_{ 4} + x_{ 6} +  & x_{ 8} + x_{10}\\
	x_{ 5} + x_{ 7} +  & x_{ 8} + x_{ 9}	&	
\end{align*}

	The associated linear program is infeasible.		
\end{proof}

\begin{theorem}
	$g(4, 15) = \binom{14}{3} = 364$.
\end{theorem}

\begin{proof}
The following sums generated by \PropagateNegative\ (Algorithm~\ref{alg:propagation})  must be strictly negative or we have at least 364 nonnegative sets:
\begin{align*}
	x_{ 1} + x_{ 8} +  & x_{14} + x_{15}	&	x_{ 1} + x_{ 9} +  & x_{13} + x_{15}	&	x_{ 1} + x_{11} +  & x_{12} + x_{15}\\
	x_{ 2} + x_{ 5} +  & x_{11} + x_{15}	&	x_{ 2} + x_{ 6} +  & x_{10} + x_{15}	&	x_{ 2} + x_{ 6} +  & x_{11} + x_{14}\\
	x_{ 2} + x_{ 7} +  & x_{ 9} + x_{15}	&	x_{ 2} + x_{ 7} +  & x_{10} + x_{14}	&	x_{ 2} + x_{ 8} +  & x_{12} + x_{13}\\
	x_{ 2} + x_{10} +  & x_{11} + x_{13}	&	x_{ 3} + x_{ 4} +  & x_{13} + x_{15}	&	x_{ 3} + x_{ 5} +  & x_{ 9} + x_{15}\\
	x_{ 3} + x_{ 5} +  & x_{10} + x_{14}	&	x_{ 3} + x_{ 6} +  & x_{ 8} + x_{15}	&	x_{ 3} + x_{ 6} +  & x_{ 9} + x_{14}\\
	x_{ 3} + x_{ 6} +  & x_{10} + x_{13}	&	x_{ 3} + x_{ 8} +  & x_{ 9} + x_{13}	&	x_{ 3} + x_{ 9} +  & x_{11} + x_{12}\\
	x_{ 4} + x_{ 5} +  & x_{12} + x_{13}	&	x_{ 4} + x_{ 6} +  & x_{ 9} + x_{13}	&	x_{ 4} + x_{ 7} +  & x_{ 8} + x_{14}\\
	x_{ 4} + x_{ 7} +  & x_{10} + x_{12}	&	x_{ 5} + x_{ 6} +  & x_{ 8} + x_{14}	&	x_{ 5} + x_{ 8} +  & x_{ 9} + x_{12}\end{align*}

The following sums generated by \PropagatePositive\ (Algorithm~\ref{alg:propagatepos}) must be nonnegative or else the associated linear program becomes infeasible:
\begin{align*}
	x_{ 3} + x_{ 4} + &x_{ 7} + x_{10}	&	x_{ 3} + x_{ 4} + &x_{ 6} + x_{11}	&	x_{ 2} + x_{ 4} + &x_{ 7} + x_{11}\\
	x_{ 1} + x_{ 7} + &x_{ 9} + x_{12}	&	x_{ 1} + x_{ 6} + &x_{ 8} + x_{13}	&	x_{ 1} + x_{ 5} + &x_{11} + x_{12}\\
	x_{ 1} + x_{ 5} + &x_{ 8} + x_{14}	&	x_{ 1} + x_{ 4} + &x_{11} + x_{13}	&	x_{ 1} + x_{ 4} + &x_{ 7} + x_{15}\\
	x_{ 1} + x_{ 3} + &x_{ 8} + x_{15}	&	x_{ 1} + x_{ 2} + &x_{11} + x_{15}	&	
\end{align*}
These positive sets now force at least 267 nonnegative $4$-sums.

The following sums generated by \PropagateNegative\ (Algorithm~\ref{alg:propagation}) must be strictly negative or we have 364 nonnegative sets:
\begin{align*}
	x_{ 1} + x_{ 3} +  & x_{13} + x_{15}	&	x_{ 1} + x_{ 4} +  & x_{ 9} + x_{15}	&	x_{ 1} + x_{ 4} +  & x_{11} + x_{14}\\
	x_{ 1} + x_{ 5} +  & x_{ 8} + x_{15}	&	x_{ 1} + x_{ 5} +  & x_{ 9} + x_{14}	&	x_{ 1} + x_{ 5} +  & x_{10} + x_{13}\\
	x_{ 1} + x_{ 6} +  & x_{ 8} + x_{14}	&	x_{ 1} + x_{ 6} +  & x_{ 9} + x_{13}	&	x_{ 1} + x_{ 7} +  & x_{11} + x_{12}\\
	x_{ 1} + x_{ 8} +  & x_{10} + x_{12}	&	x_{ 2} + x_{ 3} +  & x_{ 9} + x_{15}	&	x_{ 2} + x_{ 3} +  & x_{10} + x_{14}\\
	x_{ 2} + x_{ 4} +  & x_{ 8} + x_{14}	&	x_{ 2} + x_{ 4} +  & x_{ 9} + x_{13}	&	x_{ 2} + x_{ 5} +  & x_{ 7} + x_{14}\\
	x_{ 2} + x_{ 5} +  & x_{ 8} + x_{13}	&	x_{ 2} + x_{ 5} +  & x_{ 9} + x_{12}	&	x_{ 2} + x_{ 6} +  & x_{ 8} + x_{12}\\
	x_{ 2} + x_{ 6} +  & x_{10} + x_{11}	&	x_{ 2} + x_{ 7} +  & x_{ 9} + x_{11}	&	x_{ 3} + x_{ 4} +  & x_{ 7} + x_{15}\\
	x_{ 3} + x_{ 4} +  & x_{10} + x_{12}	&	x_{ 3} + x_{ 5} +  & x_{ 6} + x_{15}	&	x_{ 3} + x_{ 5} +  & x_{ 7} + x_{13}\\
	x_{ 3} + x_{ 5} +  & x_{ 8} + x_{12}	&	x_{ 3} + x_{ 5} +  & x_{10} + x_{11}	&	x_{ 3} + x_{ 6} +  & x_{ 8} + x_{11}\\
	x_{ 4} + x_{ 5} +  & x_{ 9} + x_{11}	&	x_{ 4} + x_{ 6} +  & x_{ 7} + x_{12}	&	x_{ 4} + x_{ 8} +  & x_{ 9} + x_{10}\\
	x_{ 5} + x_{ 7} +  & x_{ 9} + x_{10}	&	\end{align*}

The associated linear program is infeasible.
\end{proof}

\begin{theorem}
	$g(4, 17) = \binom{16}{3} = 560$.
\end{theorem}

\begin{proof}
The following sums generated by \PropagateNegative\ (Algorithm~\ref{alg:propagation}) must be strictly negative or we have at least 560 nonnegative sets:
\begin{align*}
	x_{ 1} + x_{ 6} +  & x_{16} + x_{17}	&	x_{ 1} + x_{ 7} +  & x_{13} + x_{17}	&	x_{ 1} + x_{ 8} +  & x_{12} + x_{17}\\
	x_{ 1} + x_{ 8} +  & x_{15} + x_{16}	&	x_{ 1} + x_{ 9} +  & x_{13} + x_{16}	&	x_{ 2} + x_{ 5} +  & x_{15} + x_{17}\\
	x_{ 2} + x_{ 6} +  & x_{12} + x_{17}	&	x_{ 2} + x_{ 6} +  & x_{14} + x_{16}	&	x_{ 2} + x_{ 7} +  & x_{11} + x_{17}\\
	x_{ 2} + x_{ 7} +  & x_{12} + x_{16}	&	x_{ 2} + x_{ 8} +  & x_{10} + x_{17}	&	x_{ 2} + x_{ 8} +  & x_{11} + x_{16}\\
	x_{ 2} + x_{ 8} +  & x_{13} + x_{15}	&	x_{ 2} + x_{ 9} +  & x_{12} + x_{15}	&	x_{ 3} + x_{ 5} +  & x_{11} + x_{17}\\
	x_{ 3} + x_{ 5} +  & x_{13} + x_{16}	&	x_{ 3} + x_{ 6} +  & x_{10} + x_{16}	&	x_{ 3} + x_{ 6} +  & x_{12} + x_{15}\\
	x_{ 3} + x_{ 7} +  & x_{ 9} + x_{17}	&	x_{ 3} + x_{ 7} +  & x_{11} + x_{15}	&	x_{ 3} + x_{ 7} +  & x_{13} + x_{14}\\
	x_{ 3} + x_{ 8} +  & x_{10} + x_{15}	&	x_{ 3} + x_{ 8} +  & x_{12} + x_{14}	&	x_{ 3} + x_{ 9} +  & x_{11} + x_{14}\\
	x_{ 4} + x_{ 5} +  & x_{10} + x_{17}	&	x_{ 4} + x_{ 5} +  & x_{11} + x_{16}	&	x_{ 4} + x_{ 6} +  & x_{ 9} + x_{17}\\
	x_{ 4} + x_{ 6} +  & x_{11} + x_{15}	&	x_{ 4} + x_{ 6} +  & x_{13} + x_{14}	&	x_{ 4} + x_{ 7} +  & x_{ 9} + x_{16}\\
	x_{ 4} + x_{ 7} +  & x_{10} + x_{15}	&	x_{ 4} + x_{ 7} +  & x_{11} + x_{14}	&	x_{ 4} + x_{ 8} +  & x_{ 9} + x_{15}\\
	x_{ 4} + x_{ 8} +  & x_{10} + x_{14}	&	x_{ 4} + x_{ 9} +  & x_{12} + x_{13}	&	x_{ 5} + x_{ 6} +  & x_{ 9} + x_{16}\\
	x_{ 5} + x_{ 6} +  & x_{10} + x_{15}	&	x_{ 5} + x_{ 6} +  & x_{12} + x_{14}	&	x_{ 5} + x_{ 7} +  & x_{ 8} + x_{17}\\
	x_{ 5} + x_{ 7} +  & x_{ 9} + x_{15}	&	x_{ 5} + x_{ 7} +  & x_{10} + x_{14}	&	x_{ 5} + x_{ 8} +  & x_{11} + x_{13}\\
	x_{ 6} + x_{ 7} +  & x_{12} + x_{13}	&	x_{ 7} + x_{ 9} +  & x_{10} + x_{13}	&	
\end{align*}

The following sums generated by \PropagatePositive\ (Algorithm~\ref{alg:propagatepos}) must be nonnegative or else the associated linear becomes infeasible:
\begin{align*}
	x_{ 4} + x_{ 6} + &x_{ 9} + x_{11}	&	x_{ 4} + x_{ 6} + &x_{ 7} + x_{13}	&	x_{ 4} + x_{ 5} + &x_{10} + x_{11}\\
	x_{ 4} + x_{ 5} + &x_{ 9} + x_{14}	&	x_{ 4} + x_{ 5} + &x_{ 6} + x_{15}	&	x_{ 3} + x_{ 7} + &x_{ 9} + x_{11}\\
	x_{ 3} + x_{ 6} + &x_{ 9} + x_{13}	&	x_{ 3} + x_{ 6} + &x_{ 7} + x_{14}	&	x_{ 3} + x_{ 5} + &x_{11} + x_{12}\\
	x_{ 3} + x_{ 5} + &x_{10} + x_{14}	&	x_{ 3} + x_{ 5} + &x_{ 7} + x_{15}	&	x_{ 3} + x_{ 4} + &x_{11} + x_{13}\\
	x_{ 3} + x_{ 4} + &x_{ 9} + x_{15}	&	x_{ 3} + x_{ 4} + &x_{ 6} + x_{16}	&	x_{ 2} + x_{ 8} + &x_{10} + x_{12}
\end{align*}
These positive sets now force at least 560 nonnegative $4$-sums, and our target was $560$ nonnegative $4$-sums.
\end{proof}

\end{document}